\documentclass[10pt]{amsart}
\usepackage{latexsym}
\usepackage{amsmath}
\usepackage{amsfonts}
\usepackage{amsthm}
\usepackage{amssymb}
\usepackage[T1]{fontenc}
\begin{document}

\title[approximate solutions of Sincov's equation]{Richard's inequality, Cauchy-Schwarz's inequality and approximate solutions of Sincov's equation}

\author{W\l{}odzimierz~Fechner}
\address{Institute of Mathematics, \L\'od\'z University of Technology, ul. W\'olcza\'nska 215, 90-924 \L\'od\'z, Poland}
\email{wlodzimierz.fechner@p.lodz.pl}

\newtheorem{thm}[]{Theorem}
\newtheorem{cor}[]{Corollary}
\newtheorem{lem}[]{Lemma}
\newtheorem{prop}[]{Proposition}
\theoremstyle{remark}
\newtheorem{rem}[]{Remark}
\newtheorem{ex}[]{Example}
\newcommand{\N}{\mathbb{N}}
\newcommand{\R}{\mathbb{R}}
\newcommand{\C}{\mathbb{C}}
\newcommand{\K}{\mathbb{K}}
\renewcommand{\(}{\left(} \renewcommand{\)}{\right)}
\renewcommand{\[}{\left[} \renewcommand{\]}{\right]}

\keywords{Richard inequality; Cauchy-Schwarz inequality; Sincov equation}
\subjclass[2010]{26D15, 39B62, 39B82, 46C05}

\begin{abstract}
We observe a connection between Cauchy-Schwarz' and Richard's inequalities in inner product spaces and a Ulam-type stability problem for multiplicative Sincov's functional equation. We prove that this equation is super-stable for unbounded mappings, i.e. every unbounded approximate solution is an exact solution. 
\end{abstract}

\maketitle

\section{Inequalities}

Let $H$ be an inner product space over the field $\K \in \{ \R, \C \}$. In the present note we are concerned with Richard's inequality (see \cite{R}):
\begin{equation}\label{R}
\left| {\left\langle a|x\right\rangle\cdot \left\langle x|b\right\rangle}-  \frac{\left\langle a|b\right\rangle\cdot \|x\|^2}{2}\right|\leq \frac{\|a\|\cdot \|b\|\cdot \|x\|^2}{2}
\end{equation}
for all $a,b, x \in H$.
Earlier, Ch. Blatter \cite[Lemma 2]{Bl} obtained a one-sided estimate of \eqref{R} in a real inner product space, but his result remained unknown for several years.
Inequality \eqref{R} improves Buzano inequality (see \cite{B}):
\begin{equation}\label{B}
\big|\left\langle a|x\right\rangle\cdot \left\langle x|b\right\rangle\big|\leq \frac12\big[\|a\|\cdot \|b\| + |\left\langle a|b\right\rangle|\big]\cdot \|x\|^2.
\end{equation}
for all $a,b, x \in H$.
Observe that for $a=b$ inequality \eqref{B} reduces to Cauchy-Schwarz's inequality. Recently, a number of papers regarding inequalities \eqref{R} and \eqref{B} appeared. We refer the reader to a recent article by S.S. Dragomir \cite{D} and references therein.

\medskip

Let us define a map $F\colon (H \setminus \{0\} ) \times (H \setminus \{0\}) \to \K$ by the formula
$$F(u,v):=\frac{2\left\langle u|v\right\rangle}{\|u\|\cdot \|v\|}, \quad u, v \in H\setminus \{0\}.$$
With the aid of this notation inequality \eqref{R} can be rewritten as
\begin{equation}
|F(a,x)\cdot F(x,b) - F(a,b)|\leq 2, \quad a, b, x \in H\setminus \{0\},
\label{1}
\end{equation}
whereas Cauchy-Schwarz's inequality is equivalent to
\begin{equation}
 |F(u,v)|\leq 2,  \quad u, v \in H\setminus \{0\}.
\label{2}
\end{equation}
The main purpose of the paper is to study a connection between \eqref{1} and \eqref{2} in case $F$ is an arbitrary mapping defined on a non-empty set. Moreover, we will replace the number $2$ appearing on the right-hand sides of both estimates by an arbitrary non-negative constant.

\section{Stability of Sincov's functional equation}

\subsection{The problem}
Let $X$ be a non-empty set, $F\colon X \times X \to \C$ be an arbitrary mapping and assume that $c \geq 0$ is a fixed constant. 

Multiplicative Sincov's functional equation is of the form
\begin{equation}\label{S}
F(a,x)\cdot F(x,b) = F(a,b), \quad a, b, x \in X.
\end{equation}
Equation \eqref{S} has an old history and is of significant importance. For details we refer the reader to papers of D. Gronau \cite{G, G2}.

In this section we will study approximate solutions of \eqref{S}. We are concerned with the following Ulam-type stability problem:
\begin{equation}
|F(a,x)\cdot F(x,b) - F(a,b)|\leq c, \quad a, b, x \in X.
\label{main}
\end{equation}
Observe that \eqref{main} is a direct abstract version of \eqref{1} and we assume that mapping $F\colon X \times X \to \C$ is arbitrary. 

\medskip

Dealing with a Ulam-type stability problem one  wants to establish a relation between approximate solutions of a given equation and exact solutions. Several types of stability behaviour are known (main types are: stability, super-stability, hyper-stability).
For more informations on the stability of functional equations the reader is referred to a survey article of Z. Moszner \cite{M}. On page 64 of \cite{M} some variants of Sincov's equation together with their stabilities appear. 

\subsection{Bounded solutions}
In this subsection we exploit the case of bounded mappings satisfying \eqref{main}. First, we give some examples. Next, some conditions related to the boundedness of solution of \eqref{main} are provided.

\medskip

Our first easy remark shows that every bounded mapping is a solution of \eqref{main}. 

\begin{rem}\label{r1}
Assume that $X$ is a non-empty set, $F\colon X\times X \to \C$ is a bounded mapping and define $c_0:= \sup \{|F(x, y)| : x, y \in X\}.$ Then $F$ solves \eqref{main} with $c=c_0^2+c_0$. In general the value  $c=c_0^2+c_0$ cannot be replaced by a smaller constant. To see this take arbitrary $c_0>0$ and define $F\colon X\times X \to \C$ as a constant mapping equal to $-c_0$. Then, the left-hand side of \eqref{main} is equal to $c_0^2+c_0$.
\end{rem}

Next, we provide an example which shows that for a given $c>0$ there exist bounded but arbitrary large solutions of \eqref{main} which do not satisfy equation \eqref{S}. 

\begin{ex}\label{e1}
Let $c>0$ and $n \in \N$ be fixed constants. Take $X= \{ n, n+1, \ldots , n^2\}$ and define $F\colon X\times X \to \C$ by
$$F(a,b) := \frac{a}{b+c}, \quad a, b \in X.$$
Then $$\sup \{ F(a,b) : a, b \in X\} = \frac{n^2}{n+c}$$
and
$$|F(a,x)\cdot F(x,b) - F(a,b)| = \frac{ac}{(b+c)(x+c)} \leq \frac{n^2}{(n+c)^2}c <   c, \quad a, b, x \in X.$$
\end{ex}

Another example showing a similar effect in a different settings is due to Z. Moszner \cite[p. 64]{M}. His claim was that multiplicative Sincov's equation is not stable in a sense of Ulam in the class of mappings which attain positive values only. In order to justify this he observed that if $n>2$ is a positive integer such that $(1-1/n)/n\leq c$, then a constant function $F_c$ which is equal to $1/n$ satisfies \eqref{main}. Further, it is easy to observe that every solution of \eqref{S} with non-zero values is equal to $1$ on the diagonal and, as a consequence,
can't be uniformly close to $F_c$. Note however, that if we do not restrict ourselves to the class of positive mappings, then the function constantly equal to zero is a solution of \eqref{S} which uniformly approximates $F_c$. 

\medskip

Next, we will establish some basic properties of solutions of \eqref{main} related to boundedness.

\begin{prop}\label{p1}
Assume that $c\geq 0$ is fixed, $F\colon  X\times X  \to \C$ solves \eqref{main} and $a, x \in X$ are arbitrary. Then: 
\begin{itemize}
	\item[(i)] $|F(a,a) - F(x,x)|\leq 2c$,
	\item[(ii)] the value $|F(x,x)|$ is bounded by a constant depending upon $F$,
	\item[(iii)] the product $|F(a,x)\cdot F(x,a)|$ is bounded by a constant depending upon $F$.
\end{itemize} 
\end{prop}
\begin{proof}
Directly from \eqref{main} we obtain
\begin{align}
|  F(a, x) F(x,a) - F(a,a)| & \leq c, \label{ax} \\
|  F(x, a) F(a,x) - F(x,x)| & \leq c,  \label{xa}
\end{align}
which proves (i). 

To justify (ii) it is enough to observe that by \eqref{ax} and \eqref{xa} we have
$$|F(x,x)|\leq |F(x,a) \cdot F(a,x) |+c \leq |F(a,a)| + 2c.$$
Now, (iii) follows from the above computation and from (ii).
\end{proof}

\begin{prop}\label{p2}
Assume that $c\geq 0$ is fixed and $F\colon  X\times X  \to \C$ solves \eqref{main}. If for a certain $x_0 \in X$ we have $F(x_0,x_0)\neq 1$, then both
functions $F(x_0, \cdot)$ and $F( \cdot, x_0 )$ are bounded.
\end{prop}
\begin{proof}
Suppose that we are given $x_0 \in X$ such that function $F(x_0, \cdot)$ is unbounded. Then apply \eqref{main} for $a=x=x_0$ to get
$$|F(x_0,b)|\cdot | F(x_0,x_0)-1|\leq c, \quad b \in X.$$
Clearly, since the value $|F(x_0,b)|$ can be as large as we want, we get $F(x_0,x_0)=1$. 

Proof in case $F( \cdot, x_0)$ is unbounded is similar.
\end{proof}

\subsection{Unbounded solutions}
In this subsection we prove that multiplicative Sincov's functional equation is super-stable. This means that unbounded mappings which satisfy inequality \eqref{main} solve equation \eqref{S} and, as a consequence are of some prescribed form.

\begin{lem}\label{l1}
Assume that $c\geq 0$ is fixed, $F\colon X\times X \to \C$ is a solution of \eqref{main} and there exists a $y_0\in X$ such that the map $F(\cdot, y_0)$ is unbounded. Then $F$ never vanishes and for every $y\in X$ the map $F(\cdot, y)$ is unbounded.
\end{lem}
\begin{proof}
Fix a sequence $(a_n)_{n \in \N}\subset X$ such that $|F(a_n,y_0)|\to \infty$ and an arbitrary point $y \in X$. From \eqref{main} we have 
$$|F(a_n,y) \cdot F(y, y_0) - F(a_n,y_0)|\leq c.$$ 
Thus $|F(a_n,y) \cdot F(y, y_0)|\to \infty$, which implies that $|F(a_n,y)|\to \infty$ and $F(y, y_0)\neq 0$. To get that $F$ never vanishes it is enough to repeat the foregoing reasoning and apply the already proven part with $y_0$ and $y$ replaced by two arbitrary points.
\end{proof}

The next lemma can be proven analogously.

\begin{lem}\label{l2}
Assume that $c\geq 0$ is fixed, $F\colon X\times X \to \C$ is a solution of \eqref{main} and there exists a $x_0\in X$ such that the map $F(x_0, \cdot)$ is unbounded. Then $F$ never vanishes and for every $x\in X$ the map $F(x, \cdot)$ is unbounded.
\end{lem}

Our next example shows that assumptions of both previous lemmas are independent.

\begin{ex}\label{e0}
Let $X=[1, + \infty)$ and define $F\colon X \times X \to \C$ as $F(x,y) = \frac{x}{y}$. Then $F$ solves \eqref{main} with every $c\geq 0$, for every $x_0 \in X$ the map $F(x_0, \cdot)$ is bounded and for every $y_0 \in X$ the map $F(\cdot, y_0)$ is unbounded.
\end{ex}

Now, we are ready to state and prove our main theorem.

\begin{thm}\label{t1}
Assume that $c\geq 0$ is fixed, $F\colon X\times X \to \C$ is a solution of \eqref{main} and there exist $x_0, y_0 \in X$ such that the maps $F(x_0, \cdot), F(\cdot, y_0)$ are unbounded. Then, there exists a mapping $f\colon X \to \C\setminus \{0\}$ such that
\begin{equation}\label{eq}
 F(u, v) = \frac{f(u)}{f(v)}, \quad u, v \in X.
\end{equation}
\end{thm}
\begin{proof}
Define $f, g \colon X \to \C$ as $f:= F(\cdot , x_0)$ and $g:= F(x_0, \cdot)$.
In view of Lemma \ref{l1} and Lemma \ref{l2} both mappings are unbounded and never vanish.
Apply \eqref{main} with $x=x_0$ to get
\begin{equation}\label{est0}
\left| F(a, b) - f(a)g(b)\right|\leq c, \quad a, b \in X.
\end{equation}

Our first claim is that $f(x)g(x)=1$ for every $x \in X$. In order to show this, fix arbitrary $a, b , x \in X$.
Apply \eqref{est0} to obtain
$$
| F(a,x)f(x)g(b) - f(a)g(x)f(x)g(b)|\leq c |f(x)|\cdot |g(b)|
$$
and
$$
|F(a,x)F(x,b) - F(a,x)f(x)g(b)|\leq c|F(a,x)|\leq c^2+c|f(a)|\cdot |g(x)|.
$$
Adding these inequalities together we get
\begin{equation}\label{i}
|F(a,x)F(x,b) - f(a)g(x)f(x)g(b)|\leq  c^2+c( |f(x)|\cdot |g(b)| + |f(a)|\cdot |g(x)|).
\end{equation}
Using the triangle inequality and then \eqref{i}, \eqref{main} and \eqref{est0} we derive
\begin{align*}
|f(a)g(x)f(x)g(b) - f(a)g(b)| &\leq  |f(a)g(x)f(x)g(b) - F(a,x)F(x,b) | \\&+ |F(a,x)F(x,b) -F(a,b)| + |F(a,b) - f(a)g(b)|\\&\leq  c^2+ 2c + c( |f(x)|\cdot |g(b)| + |f(a)|\cdot |g(x)|).
\end{align*}
From this we have
\begin{equation}\label{dupa}
|g(x)f(x) - 1|\leq \frac{c^2+ 2c}{ |f(a)|\cdot |g(b)|} + c\frac{|f(x)|}{|f(a)|} +  c\frac{|g(x)|}{|g(b)|}.
\end{equation}
Since both mappings $f$ and $g$ are unbounded, then keeping $x$ fixed one can choose $a, b \in X$ such that the right-hand side of \eqref{dupa} is arbitrary small. This proves the desired equality $f(x)g(x)=1$. This, together with \eqref{est0} gives us the estimate.
\begin{equation}\label{est1}
\left| F(u, v) - \frac{f(u)}{f(v)}\right|\leq c, \quad u, v \in X.
\end{equation}

To finish the proof fix $u,v \in X$ arbitrarily. By Lemma \ref{l1} there exists a sequence $(a_n)_{n \in \N}\subset X$  such that $|F(a_n,u)|\to \infty$. From \eqref{main} we get 
$$\left|F(u,v) - \frac{F(a_n,v)}{F(a_n,u)} \right|\leq  \frac{c}{|F(a_n,u)|} .$$
From this we see that
$$\lim_{n \to \infty} \frac{F(a_n,v)}{F(a_n,u)} = F(u,v).$$
On the other hand, since $|F(a_n,u)|\to \infty$ and $|F(a_n,v)|\to \infty$, then by \eqref{est1} we  have
$$\lim_{n \to \infty} \frac{F(a_n,v) - \frac{f(a_n)}{f(v)}}{F(a_n,u)}  =   \lim_{n \to \infty} \frac{F(a_n,u) - \frac{f(a_n)}{f(u)}}{F(a_n,v)}  = 0.$$
Consequently, we obtain
$$F(u,v)  =   \lim_{n \to \infty} \frac{F(a_n,v)}{F(a_n,u)}=   \lim_{n \to \infty}  \frac{ \frac{f(a_n)}{f(v)} }{  \frac{f(a_n)}{f(u)} }=\frac{f(u)}{f(v)} .$$
\end{proof}

\begin{rem}
It is easy to see that function $f$ appearing in formula \eqref{eq} is determined uniquely up to a multiplicative constant. 
\end{rem}

\section{Conclusions and final remarks}

\begin{rem} Our Theorem \ref{t1} is in line with well-known results on the superstability of the exponential Cauchy equation and related equations of J.A. Baker, J. Lawrence, F. Zorzitto \cite{B1}, J.A. Baker \cite{B2}, R. Ger, P. \v{S}emrl \cite{GS}, T. Kochanek, M. Lewicki \cite{KL}, among others. However, taking into account the case of bounded solutions of \eqref{main} no full analogy can be observed, see Example \ref{e1}. Therefore, we can conclude that the multiplicative Sincov's equation is super-stable, however, there is no control over bounded solutions. 
\end{rem}

The following example is motivated by an example due to J.A. Baker from \cite{B2} and shows that in higher dimension the super-stability effect for Sincov's equation does not hold.

\begin{ex}
Denote by $M_{2\times 2}(\R)$ the algebra of all real $2 \times 2$ matrices equipped with the usual matrix norm. Take $X=(0, \infty)$ and fix a constant $c>0$. Choose $c_0>0$ such that $|c_0^2-c_0|=c$ and let $F\colon X \times X \to M_{2\times 2}(\R)$ be defined as
$$F(u,v) = 
\left[\begin{array}{cc}
	\frac{u}{v} & 0 \\ 0 & c_0
\end{array}\right], \quad u, v \in X.$$
It is easy to check that 
$$ \left\| F(a,x)\cdot F(x,b) - F(a,b) \right\| =  |c_0^2-c_0|=c \quad a, b, x \in X.$$ Therefore $F$ does not solve equation \eqref{S}, is unbounded and it satisfies the stability inequality \eqref{main}.
\end{ex}

\section{Acknowledgements}
The work performed in this study  was supported by the National Science Centre, Poland (under Grant No. 2015/19/B/ST6/03259).


\begin{thebibliography}{99}

\bibitem{B1} J.A. Baker, J. Lawrence, F. Zorzitto, \emph{The stability of the equation $f (x + y) = f (x)f (y)$}, Proc. Amer. Math. Soc., 74 (1979), 242--246.

\bibitem{B2} J.A. Baker, \emph{The stability of the cosine equation}, Proc. Amer. Math. Soc. 80 (1980), 411--416.

\bibitem{Bl}
Ch. Blatter, \emph{Zur Riemannschen Geometrie im Grossen auf dem M\"obiusband},
Compositio Math., 15 (1962--1964),  88--107. 

\bibitem{B} {M.L.  Buzano},  \emph{Generalizzazione  della  disuguglianza  di  Cauchy-Schwarz} (Italian),
Rend.  Sem. Mat. Univ. e Politech. Torino, 31 (1971/73), 405--409 (1974).

\bibitem{D} S.S. Dragomir, \emph{A Buzano type inequality for two Hermitian forms and applications}, Linear and Multilinear Algebra, 65:3 (2017), 514--525.

\bibitem{GS} R. Ger, P. \v{S}emrl, \emph{The stability of the exponential equation}, Proc. Amer. Math. Soc., 124 (1996), 779--787.

\bibitem{G} {D. Gronau}, \emph{A remark on Sincov's functional equation}, Notices of the South African Mathematical Society 31, No. 1, April 2000, 1--8.

\bibitem{G2} {D. Gronau}, \emph{Translation equation and Sincov's equation - A historical remark}, ESAIM: Proceedings and Surveys, November 2014, Vol. 46, 43--46.

\bibitem{KL} T. Kochanek, M. Lewicki, \emph{Stability problem for number-theoretically multiplicative functions}, Proc. Amer. Math. Soc., 135 (2007), 2591--2597.

\bibitem{M} Z. Moszner, \emph{On the stability of functional equations}, Aequationes Math., 77 (2009), 33--88.

\bibitem{R} {U. Richard}, \emph{Sur  des  in\'egalit\'es  du  type  Wirtinger  et  leurs  application  aux  \'equationes diff\'erentielles ordinaires}, Colloquium of Analysis held in Rio de Janeiro, August  1972,  pp.  233--244.

\end{thebibliography}
\end{document}